\documentclass[12 pt]{amsart}
\usepackage{amsmath,times,epsfig,amssymb,amsbsy,amscd,amsfonts,amstext,color,bm}
\usepackage[arrow,matrix]{xy}

\usepackage{graphicx}

\theoremstyle{plain}
\newtheorem{theorem}{Theorem}[section]
\newtheorem{corollary}[theorem]{Corollary}
\newtheorem{lemma}[theorem]{Lemma}

\newtheorem{proposition}[theorem]{Proposition}

\makeatletter
    
    \@addtoreset{equation}{section}
  \makeatother

\theoremstyle{definition}
\newtheorem{definition}[theorem]{Definition}
\newtheorem{example}[theorem]{Example}

\theoremstyle{remark}
\newtheorem{remark}[theorem]{Remark}

\newcommand{\R}{\mathbb{R}}

\newcommand{\Z}{\mathbb{Z}}

\newcommand{\Tor}{\operatorname{Tor}}

\newcommand{\simp}{\operatorname{simp}}
\newcommand{\Tot}{\operatorname{Tot^\oplus}}
\newcommand{\id}{\operatorname{id}}
\newcommand{\Int}{\operatorname{Int}}

\begin{document}

\title[Magnitude homology]{Magnitude homology of metric spaces and order complexes}

\date{\today}

\begin{abstract}
Hepworth, Willerton, 
Leinster and Shulman introduced the magnitude homology groups 
for enriched categories, in particular, for metric spaces. 
The purpose of this paper is to describe the 
magnitude homology group of a metric space in terms of order 
complexes of posets. 

In a metric space, an interval (the set of points between two chosen points) 
has a natural poset structure, which is called the interval poset. 
Under additional assumptions on sizes of $4$-cuts, we show that 
the magnitude chain complex can be 
constructed using tensor products, direct sums and degree shifts 
from order complexes of interval posets. 

We give several applications. First, we show the vanishing of higher magnitude 
homology groups for convex subsets of the Euclidean space. Second, magnitude 
homology groups carry the information about the diameter of a hole. 
Third, we construct a finite graph whose 
$3$rd magnitude homology group has torsion. 
\end{abstract}

\author{Ryuki Kaneta}
\address{Ryuki Kaneta, Department of Mathematics, Hokkaido University, Kita 10, Nishi 8, Kita-Ku, Sapporo 060-0810, Japan.}
\email{s173008@math.sci.hokudai.ac.jp}
\author{Masahiko Yoshinaga}
\address{Masahiko Yoshinaga, Department of Mathematics, Faculty of Science, Hokkaido University, Kita 10, Nishi 8, Kita-Ku, Sapporo 060-0810, Japan.}
\email{yoshinaga@math.sci.hokudai.ac.jp}



\date{\today}
\maketitle

\tableofcontents

\section{Introduction}
\label{sec:intro}

The magnitude for an enriched category was introduced by Leinster 
\cite{lei-eul, lei-mag} as a generalization of the Euler characteristic 
of the nerve of a category. The same notion for finite metric spaces 
was also studied in theoretical ecology as a measure of biological diversity 
\cite{so-po}. 

The magnitude homology groups,  a ``categorification'' of magnitudes, 
for finite graphs were defined by 
Hepworth and Willerton \cite{hep-wil} and for general setting 
by Leinster and Shulman \cite{Lei-Shu} recently. 
The magnitude homology groups for a finite metric space recover 
its magnitude as a divergent alternating sum which coincide with 
the divergent series studied in \cite{ber-lei}. 

The magnitude homology group $H_n^{\Sigma, \ell}(X)$ of a 
metric space $X$ is an abelian group bigraded by 
a non-negative integer (degree) $n$ and 
a non-negative real number (grading) $\ell$. 
Leinster and Shulman proved that magnitude homology groups 
detect several geometric properties of the space $X$. Among others, 
$X$ is Menger convex (see \S \ref{subsec:notion} for the definition) 
if and only if $H_1^{\Sigma, \ell}(X)=0$ for all $\ell>0$ 
\cite[Corollary 7.6]{Lei-Shu}. 
In particular, if 
$X$ is a closed subset of the Euclidean space $\R^N$, 
$X$ is convex in the usual sense if and only if 
$H_1^{\Sigma, \ell}(X)=0$ for all $\ell>0$. 
They also posed number of interesting open problems on 
magnitude homology of metric spaces \cite[\S 8]{Lei-Shu}. 

The purpose of this paper is to develop methods of computing 
magnitude homology groups of metric spaces (under certain 
assumptions on $X$ and $\ell$). 
We reduce the magnitude chain complex to the order 
complexes of interval posets. 
Our reduction proceeds in the following three steps. 
\begin{itemize}
\item[1.] 
The magnitude chain complex is decomposed into a direct sum of 
\emph{framed magnitude chain complexes}. (\S \ref{sec:frame})
\item[2.] 
The framed magnitude chain complex is decomposed into tensor 
product of those of intervals. (\S \ref{sec:tensor})
\item[3.] 
The framed magnitude chain complex of an 
interval is isomorphic (up to degree shift) 
to the order complex of the \emph{interval poset}. 
(\S \ref{sec:ordercpx})
\end{itemize}
Then we give several applications based on the above description. 

First, in \S \ref{sec:geodetic}, for a geodetic metric space, 
we show that the magnitude homology $H_n^{\Sigma, \ell}(X)$ is 
freely generated by certain frames (\emph{thin frames}, see 
Definition \ref{def:thin}), which generalizes 
several results by Leinster and Shulman 
\cite[Corollary 7.6, Theorem 7.25]{Lei-Shu}. 
As a result, convex subsets and open subsets $X$ of the Euclidean space 
$\R^N$ 
(more generally, Menger convex geodetic metric space with no $4$-cuts) 
has $H_n^{\Sigma, \ell}(X)=0$ for any $n>0$ and $\ell>0$. 

Second, in \S \ref{sec:diameter}, we show that the quantity 
\begin{equation}
\sup
\left\{
\left.
\frac{\ell}{n}
\right| 
n>0, \ell>0, 
H_n^{\Sigma, \ell}(X)\neq 0
\right\}
\end{equation}
carries certain geometric information about $X$. 
For example, if $X=\R^N\smallsetminus U$ 
is the complement of an open convex subset $U\subset\R^N$, then 
the above quantity is exactly equal to the diameter of $U$ 
(Theorem \ref{thm:hole}). 

Third, in \S \ref{sec:embed}, we show that the homology groups 
of the order complex of a ranked poset $P$ (of rank $r\geq 0$) 
is embedded into the magnitude homology group of the 
Hasse diagram of $\widehat{P}=P\sqcup\{\widehat{0}, \widehat{1}\}$ 
with grading $\ell=r+2$. In particular, if $P$ is the face 
poset of a triangulation of $\mathbb{RP}^2$, then the magnitude 
homology $H_3^{\Sigma, 4}(\widehat{P})$ has a $2$-torsion. 
This answers to the question raised in \cite[\S 8 (7)]{Lei-Shu} and 
\cite[\S 1.2.2]{hep-wil}. 

After the completion of the present paper, we learned that 
Beno\^it Jubin has independently obtained similar descriptions for 
the magnitude homology groups \cite{jubin}.

\section{Definition of the magnitude homology}
\label{sec:main}

\subsection{Notions on metric spaces}
\label{subsec:notion}

Let $(X, d)$ be a metric space. We begin by fixing some terminology 
which are necessary to define the notion of magnitude homology 
(\cite{Lei-Shu}). 
We say that 
the point $y\in X$ is between $x$ and $z\in X$ if $d(x, y)+d(y, z)=d(x, z)$, 
which is denoted by $x\preceq y \preceq z$. Moreover, if $x\neq y\neq z$, 
we denote $x\prec y \prec z$. 

The metric space $(X, d)$ is said to be \emph{Menger convex} if for 
any $x\neq z$ there exists a point $y\in X$ with $x\prec y\prec z$. 

We say that a tuple $\bm{x}=(x_0, x_1, \dots, x_n)\in X^{n+1}$ is a proper chain 
if $x_{i-1}\neq x_i$ for all $i=1, \dots, n$. We call $n$ the degree of 
$\bm{x}$ and define its length by 
\[
|\bm{x}|=
d(x_0, x_1)+d(x_1, x_2)+\dots+d(x_{n-1}, x_n). 
\]
We denote by $P_n(X)$ (respectively $P_n^\ell(X)$) 
the set of all proper chains of degree $n$ 
(respectively, with length $\ell$). 

The following will be used frequently. Since it is straightforward, 
we omit the proof. 
\begin{proposition}
\label{prop:easy}
Let $\bm{x}=(x_0, x_1, \dots, x_n)\in P_n(X)$. If 
$|\bm{x}|=d(x_0, x_n)$, 
then 
\begin{equation}
\label{eq:triang}
d(x_i, x_j)=d(x_i, x_{i+1})+\dots+d(x_{j-1}, x_j), 
\end{equation}
for any $0\leq i<j\leq n$. Furthermore, for any subsequence 
$i_0<i_1<\cdots <i_m$ of $\{0, 1, 2, \dots, n\}$, we have 
\begin{equation}
\label{eq:triang02}
d(x_{i_0}, x_{i_m})=d(x_{i_0}, x_{i_1})+d(x_{i_1}, x_{i_2})+\dots+
d(x_{i_{m-1}}, x_{i_m}). 
\end{equation}
\end{proposition}

\begin{definition}
\label{def:4cut}
$(x_0, x_1, x_2, x_3)\in P_3(X)$ is a \emph{$4$-cut} of $X$ if 
$x_0\prec x_1\prec x_2$, 
$x_1\prec x_2\prec x_3$ and 
$d(x_0, x_3)<|(x_0, x_1, x_2, x_3)|$. 

We denote by $m_X\geq 0$ the infimum of lengths of $4$-cuts, namely, 
\[
m_X:=
\inf\{|(x_0, x_1, x_2, x_3)|; 
(x_0, x_1, x_2, x_3) \mbox{ is a $4$-cut of $X$}\}. 
\]
In case $X$ does not have $4$-cuts, we suppose $m_X=+\infty$. 
\end{definition}

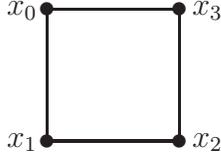
\begin{figure}[htbp]
\begin{picture}(100,50)(0,0)
\thicklines

\multiput(0,0)(50,0){2}{\circle*{5}}
\multiput(0,50)(50,0){2}{\circle*{5}}
\multiput(0,0)(50,0){2}{\line(0,1){50}}
\multiput(0,0)(0,50){2}{\line(1,0){50}}

\put(-15,48){$x_0$}
\put(-15,-2){$x_1$}
\put(55,-2){$x_2$}
\put(55,48){$x_3$}

\end{picture}
\caption{A $4$-cut $(x_0, x_1, x_2, x_3)$}\label{fig:4cut}
\end{figure}

\begin{example}
\begin{itemize}
\item[(1)] 
Let $(X, d)$ be a metric space and $Y\subset X$ be a subspace 
with restricted metric $d|_Y$. Then $m_X\leq m_Y$. 
\item[(2)] 
The Euclidean space $\R^n$ does not have $4$-cuts. 
Hence for any metric subspace 
$X\subset\R^n$, $m_X=+\infty$. 
\item[(3)] 
Let $X=S^1$ be the circle of radius $r$ with the geodesic metric. 
Then $m_X=\pi r$ (which is the distance between antipodal points). 
\end{itemize}
\end{example}

\begin{lemma}
\label{lem:global}
Let $n>1$ and $\bm{x}=(x_0, x_1, \dots, x_n)\in P_n(X)$. 
Suppose that $x_{i-1}\prec x_i\prec x_{i+1}$ for $i=1, \dots, n-1$ and 
$|\bm{x}|<m_X$. Then 
\begin{equation}
\label{eq:global}
|\bm{x}|=d(x_0, x_n). 
\end{equation}
\end{lemma}
\begin{proof}
We can prove by induction on $n$. The case $n=2$ is trivial. 
If $n=3$, (\ref{eq:global}) holds because there does not 
exist $4$-cut with length less than $m_X$. 

Let $n>3$. By inductive assumption, we have 
$d(x_0, x_{n-1})=\sum_{i=1}^{n-1} d(x_{i-1}, x_i)$. 
Since $(x_0, x_{n-2}, x_{n-1}, x_n)$ is not a $4$-cut, 
(\ref{eq:global}) holds. 
\end{proof}

\subsection{Interval poset}
\label{subsec:interval}

\begin{definition}
Let $X$ be a metric space and $a,b \in X$. 
Denote the set of points between $a$ and $b$ by 
\begin{equation*}
I_X(a,b):=\{x \in X \mid a\prec x \prec b \}. 
\end{equation*}
The set $I_X(a, b)$ carries a natural poset structure defined by 
\begin{equation*}
x\leq y\Longleftrightarrow a \prec x \preceq y
\end{equation*}
for $x, y\in I_X(a, b)$. 
(Note that the above definition is equivalent to $x\preceq y\prec b$.) 
We call $I_X(a, b)$ the \emph{interval poset} between $a$ and $b$. 
\end{definition}
As \cite[Definition 7.20]{Lei-Shu}, $X$ is said to be \emph{geodetic} 
if $I_X(a, b)$ is totally ordered (or empty) for any $a, b\in X$.

\subsection{Magnitude homology in grading $\ell$}

Next we recall the definition of the normalized chain complex 
$B_{\bullet}^{\ell}(X)$ of $X$ in grading $\ell$. 

\begin{definition}(\cite{Lei-Shu}, Lemma 7.1.)
\label{def:magnitudechaincpx}
Let $\ell\in\R_{\geq 0}$. 
The chain complex $(B_{\bullet}^{\ell}(X), \partial_{\bullet})$ is 
defined as follows. 
\begin{equation*}
B_{n}^{\ell}(X)=
\bigoplus_{\bm{x}\in P_n^\ell(X)} 
\Z \cdot\langle \bm{x} \rangle ,
\end{equation*}
The boundary map $\partial_{n}$ is defined by 
$\partial_{n}=\sum_{i=0}^n(-1)^i\partial_{n,i}$, 
where the map $\partial_{n,i}$ discards $x_i$ if 
$x_{i-1}\prec x_i \prec x_{i+1}$. 
More precisely, 
\begin{equation*}
\partial_{n,i} (\langle x_0,\cdots ,x_n \rangle)=
\begin{cases}
\langle x_0,\cdots, \widehat{x_{i}}, \cdots ,x_n \rangle, 
 & x_{i-1}\prec x_i \prec x_{i+1}\\
0, & \mbox{otherwise}, 
\end{cases}
\end{equation*}
where $\widehat{x_i}$ indicates that $x_i$ has been omitted. 
It is also denoted by $\bm{x}\smallsetminus\{x_i\}$. 

The homology $H_n^{\Sigma, \ell}(X):=H_n(B_{\bullet}^{\ell}(X))$ 
of the chain complex $(B_{\bullet}^{\ell}(X),\partial_{\bullet})$ 
is called the magnitude homology group of $X$ 
in grading $\ell$ (\cite{Lei-Shu}). 
\end{definition}

\begin{remark}
From the definition, it is easily seen that $H_0^{\Sigma, 0}(X)=\Z^{\oplus X}$, 
$H_0^{\Sigma, \ell}(X)=0$ (for $\ell>0$), and  
$H_n^{\Sigma, 0}(X)=0$ (for $n>0$). In the sequel, we are mainly interested 
in $H_n^{\Sigma, \ell}(X)$ for $n\in\Z_{>0}$ and $\ell\in\R_{>0}$. 
\end{remark}

\section{Framed magnitude homology}

\label{sec:frame}

\subsection{Frames}

\begin{definition}
Consider $\bm{x}=(x_0,x_1,\cdots, x_{n}) \in P_n(X)$. 
If $1\leq i\leq n-1$ and 
$x_{i-1}\prec x_i\prec x_{i+1}$, 
we say that the $i$-th point $x_i$ 
is a \emph{smooth point} of $\bm{x}$. 
Otherwise, we say that $x_i$ is a 
\emph{singular point} of $\bm{x}$. 
\end{definition}
The following is straightforward. 

\begin{proposition}
Let $\bm{x}=(x_0, \dots, x_n)\in P_n(X)$. 
\begin{itemize}
\item[(1)] 
Let $i\in\{1, \dots, n-1\}$. Then the following are equivalent. 
\begin{itemize}
\item 
The $i$-th point $x_i$ is a smooth point of $\bm{x}$. 
\item 
$d(x_{i-1}, x_{i+1})=
d(x_{i-1}, x_{i})+d(x_{i}, x_{i+1})$. 
\item 
$|\bm{x}\smallsetminus\{x_i\}|=|\bm{x}|$. 
\end{itemize}
\item[(2)] 
Let $i\in\{0, 1, , \dots, n\}$. Then the following are equivalent. 
\begin{itemize}
\item 
The $i$-th point $x_i$ is a singular point of $\bm{x}$. 
\item 
$i\in\{0, n\}$ or 
$d(x_{i-1}, x_{i+1})<d(x_{i-1}, x_{i})+d(x_{i}, x_{i+1})$. 
\item 
$|\bm{x}\smallsetminus\{x_i\}|<|\bm{x}|$. 
\end{itemize}
\end{itemize}
\end{proposition}

\begin{definition}
\label{def:frame}
Let $\bm{x}=(x_0,x_1,\cdots, x_{n})\in P_n(X)$. Suppose $x_{i_1}, x_{i_2}, 
\dots, x_{i_m}$ are the list of all singular points of $\bm{x}$. 
(Note that $i_1=0$ and $i_m=n$.) 
Define $\varphi(\bm{x})$ by the chain consisting of singular points of 
$\bm{x}$, namely, 
\begin{equation*}
\varphi(\bm{x}):=(x_{i_1},x_{i_2},\cdots ,x_{i_m} ). 
\end{equation*}
We call $\varphi(\bm{x})$ the \emph{frame} of $\bm{x}$. 
\end{definition}

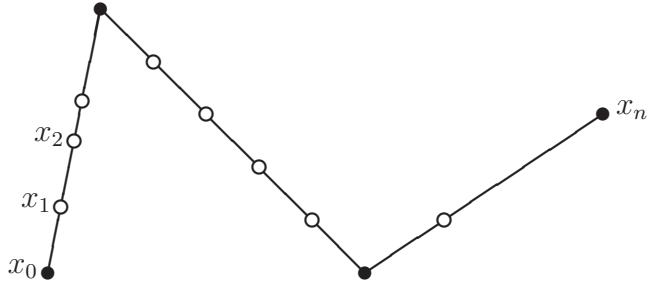
\begin{figure}[htbp]
\begin{picture}(200,100)(0,0)
\thicklines

\put(-15,0){$x_0$}

\multiput(0,0)(20,100){2}{\circle*{5}}
\put(0,0){\line(1,5){20}}

\put(-10,25){$x_1$}

{\color{white} 
\put(5,25){\circle*{5}}
}
\put(5,25){\circle{5}}

\put(-5,50){$x_2$}

{\color{white} 
\put(10,50){\circle*{5}}
}
\put(10,50){\circle{5}}

{\color{white} 
\put(13,65){\circle*{5}}
}
\put(13,65){\circle{5}}

\put(20,100){\line(1,-1){100}}

{\color{white} 
\put(40,80){\circle*{5}}
}
\put(40,80){\circle{5}}

{\color{white} 
\put(60,60){\circle*{5}}
}
\put(60,60){\circle{5}}

{\color{white} 
\put(80,40){\circle*{5}}
}
\put(80,40){\circle{5}}

{\color{white} 
\put(100,20){\circle*{5}}
}
\put(100,20){\circle{5}}

\put(120,0){\circle*{5}}

\put(120,0){\line(3,2){90}}

{\color{white} 
\put(150,20){\circle*{5}}
}
\put(150,20){\circle{5}}

\put(210,60){\circle*{5}}

\put(215,60){$x_n$}

\end{picture}
\caption{Singular points (black) and smooth points (white)}\label{fig:singular}
\end{figure}

\subsection{Geodesically simple chains}

\label{sec:geodesic}

\begin{definition}
\label{def:geodsimple}
Let $\bm{x}=(x_0, \dots x_n)\in P_n(X)$ with $\varphi(\bm{x})=
(x_{i_1},x_{i_2},\cdots ,x_{i_m} )$. We say that $\bm{x}$ is a 
\emph{geodesically simple chain} if 
\[
|(x_{i_\alpha}, x_{i_{\alpha+1}})|=
|(x_{i_\alpha}, x_{i_\alpha+1}, \dots, 
x_{i_{\alpha+1}-1}, x_{i_{\alpha+1}})|
\]
for all $\alpha=1, \dots, m-1$. 
In other words, $\bm{x}$ is geodesically simple if and only if 
$|\varphi(\bm{x})|=|\bm{x}|$. 
\end{definition}
The following is straightforward (using Lemma \ref{lem:global}). 
\begin{proposition}
\label{prop:shortgeod}
A proper chain $\bm{x}\in P_n(X)$ is geodesically simple if either 
\begin{itemize}
\item $1\leq n\leq 2$, or 
\item $n\geq 1$ and $|\bm{x}|<m_X$. 
\end{itemize}
When $n=3$, $\bm{x}=(x_0, x_1, x_2, x_3)\in P_3(X)$ with $\varphi(\bm{x})=
(x_0, x_3)$ is geodesically simple if and only if it is not a $4$-cut. 
\end{proposition}
We shall see that the set of all geodesically simple chains form a subcomplex 
of $B_\bullet^\ell(X)$ and that the boundary operator preserves the frames. 

\begin{proposition}
\label{prop:geodsimple01}
Let $\bm{x}=(x_0, \dots x_n)\in P_n(X)$ be a geodesically simple chain 
and $J=\{x_{j_1}, \dots, x_{j_p}\}$ be an arbitrary set of smooth points of 
$\bm{x}$. Let $\bm{x'}:=\bm{x}\smallsetminus\{x_{j_1}, \dots, x_{j_p}\}$. 
Then $|\bm{x'}|=|\bm{x}|$. 
\end{proposition}
\begin{proof}
Let $\{x_{i_1}, \dots, x_{i_m}\}$ be the set of all singular points as above. 
Then we have 
\[
\begin{split}
|\bm{x}|&\geq |\bm{x'}|\\
&\geq
|(x_{i_1}, \dots, x_{i_m})|. 
\end{split}
\]
By the assumption that $\bm{x}$ is a geodesically simple chain, the right 
hand side is equal to $|\bm{x}|$. 
Thus we have $|\bm{x}|= |\bm{x'}|$. 
\end{proof}

\begin{proposition}
\label{prop:geodsimplechain}
Let $\bm{x}=(x_0, \dots x_n)\in P_n(X)$ be a geodesically simple chain 
and $x_j$ be a smooth point of $\bm{x}$. 
Then $\bm{x'}:=\bm{x}\smallsetminus\{x_j\}$ is a 
geodesically simple chain with $\varphi(\bm{x'})=\varphi(\bm{x})$. 
\end{proposition}
\begin{proof}
Let $x_k$ be a smooth point of $\bm{x}$ such that $k\neq j$. 
Proposition \ref{prop:geodsimple01} implies that $x_k$ is a smooth 
point of $\bm{x'}$. 

Let $x_{i_\alpha}$ be a singular point of $\bm{x}$. Then 
\[
\begin{split}
|\bm{x}\smallsetminus\{x_{i_\alpha}, x_j\}|
&\leq
|\bm{x}\smallsetminus\{x_{i_\alpha}\}|\\
&<
|\bm{x}|\\
&=
|\bm{x}\smallsetminus\{x_j\}|. 
\end{split}
\]
Hence $x_{i_{\alpha}}$ is a singular  point of $\bm{x'}$. 
Thus we have $\varphi(\bm{x'})=\varphi(\bm{x})$. 
Then since $|\varphi(\bm{x'})|=|\varphi(\bm{x})|=|\bm{x}|=|\bm{x'}|$, 
$\bm{x'}$ is geodesically simple. 
\end{proof}
\begin{remark}
If there exists a $4$-cut, the boundary operator does not preserve 
the frame in general. 
\end{remark}

\begin{definition}
\label{def:geodsimpchaincpx}
Define $B_{n}^{\simp, \ell}(X)$ to be the 
submodule of $B_{n}^{\ell}(X)$ generated by 
geodesically simple chains. 
\end{definition}

\begin{definition}
\label{def:framedchain}
Let $F\in P_{m}(X)$. 
\[
P_n^F(X):=\{\bm{x}\in P_n^{|F|}(X)\mid \varphi(\bm{x})=F\}. 
\]
\end{definition}
Note that in case $\varphi(F)\neq F$, $P_n^F(X)=\emptyset$. 
If $\bm{x}\in P_n^F(X)$, then $\bm{x}$ is geodesically simple. 
We also have $P_n^F(X)\subset P_n^{|F|}(X)$.

\begin{definition}
Let $F\in P_{m}(X)$. Let $n\geq m$ and define 
$B_{n}^{F}(X)$ by 
\begin{equation*}
B_{n}^{F}(X)=
\bigoplus_{\bm{x}\in P_n^F(X)} 
\Z \cdot\langle \bm{x} \rangle.
\end{equation*}
\end{definition}

Clearly we have $B_n^{F}(X)\subset B_n^{\simp, |F|}(X)$, 
and moreover, 
\begin{equation}
\label{eq:decomp}
B_{n}^{\simp, \ell}(X)=
\bigoplus_{F\in P_{\leq n}^\ell(X)}
B_{n}^{F}(X), 
\end{equation}
where $P_{\leq n}(X)=P_1(X)\sqcup P_2(X)\sqcup\cdots\sqcup P_n(X)$. 
Proposition \ref{prop:geodsimple01} and Proposition \ref{prop:geodsimplechain} 
show that 
$B_{\bullet}^{\simp, \ell}(X)$ and 
$B_{\bullet}^{F}(X)$ form chain complexes. 
Thus we can define 
the \emph{magnitude homology with geodesically simple chains} by 
\[
H_n^{\simp, \ell}(X):=H_n(B_{\bullet}^{\simp, \ell}(X)), 
\]
and the \emph{framed magnitude homology with frame $F$} by 
\[
H_n^{F}(X):=H_n(B_{\bullet}^{F}(X)). 
\]
We have the following decomposition of magnitude homology 
in terms of framed magnitude homology groups. 
\begin{theorem}
\label{thm:Fdec}
Let $X$ be a metric space, $n>0$ and $\ell>0$. Then, 
\begin{itemize}
\item[(1)] 
\begin{equation*}
H_{n}^{\simp, \ell}(X) \simeq \bigoplus_{F\in P_{\leq n}^\ell(X)}H_{n}^{F}(X).
\end{equation*}
\item[(2)] 
If $n=1$, 
\begin{equation*}
H_{1}^{\Sigma, \ell}(X) 
\simeq \bigoplus_{F\in P_{1}^\ell(X)}H_{1}^{F}(X).
\end{equation*}
\item[(3)] 
Let $n>0$ and $\ell>0$. If $\ell<m_X$, then 
\begin{equation*}
H_n^{\Sigma, \ell}(X)\simeq
\bigoplus_{F\in P_{\leq n}^\ell(X)}
H_n^{F}(X). 
\end{equation*}
\end{itemize}
\end{theorem}
\begin{proof}
(1) is clear from the direct sum decomposition of 
the chain complex (\ref{eq:decomp}). 
It follows from 
Proposition \ref{prop:shortgeod} that $B_n^{\ell}(X)=B_n^{\simp, \ell}(X)$  
if either $n\leq 2$ or $\ell<m_X$. These induce (2) and (3). 
\end{proof}

\begin{remark}
Theorem \ref{thm:Fdec} (2) is just re-phrasing \cite[Corollary 7.6]{Lei-Shu}. 
\end{remark}

\section{Interval decompositions of framed magnitude homology}

\subsection{The order complex of a poset}

We first recall the notion of the order complex of a poset (see \cite{Wac} 
for further details). 
Let $P$ be a poset. Recall that the order complex $\Delta(P)$ of $P$ is 
a simplicial complex defined by 
\[
\Delta(P)=\{ (x_0, \dots, x_n) \mid n\geq 0, x_i\in P, x_0<x_1<\cdots<x_n\}. 
\]
Let $C_\bullet(P)$ be the reduced chain complex of 
the order complex. 
More precisely, define the module $C_n(P)$ by 
\begin{equation*}
C_n(P):=
\left\{
\begin{array}{cl}
\bigoplus\limits_{x_0<\cdots<x_n}
\Z \cdot\langle x_0, \cdots , x_n \rangle , 
&\mbox{ for $n\geq 0$}, \\
\Z, &
\mbox{ for $n=-1$}, 
\end{array}
\right.
\end{equation*}
and the boundary map $\partial_{n}:C_n(P)\longrightarrow 
C_{n-1}(P)$ by 
$\partial_{n}=\sum_{i=0}^n(-1)^i\partial_{n,i}$ for $n\geq 1$, 
where the map $\partial_{n,i}$ discards $x_i$, 
\begin{equation*}
\partial_{n,i} (\langle x_0,\cdots ,x_n \rangle)=
\langle x_0,\cdots,\widehat{x}_i,\cdots ,x_n \rangle.
\end{equation*}
The map $\partial_0:C_0(P)\longrightarrow\Z$ is defined by 
$\partial_0(\langle x\rangle)=1$ for all $x\in X$.

We denote the $n$-th homology $H_n(C_{\bullet}(P))$ 
by  $\widetilde{H}_n(P)$. Note that $\widetilde{H}_n(P)$ is the 
reduced homology of $\Delta(P)$. 
\begin{remark}
We may suppose $C_{-1}(P)$ is a rank one free abelian group 
$\Z \cdot\langle \emptyset \rangle$ generated by the symbol 
$\langle\emptyset\rangle$. The $(-1)$-st homology is 
\begin{equation*}
\widetilde{H}_{-1}(P)=
\left\{
\begin{array}{cc}
\Z, & \mbox{ if }P=\emptyset, \\
0, & \mbox{ if }P\neq\emptyset. 
\end{array}
\right.
\end{equation*}
We will see that the $1$-st magnitude homology $H_1^{\Sigma, \ell}(X)$ 
is generated by $(-1)$-st homology groups of interval posets. 
\end{remark}

The following will be frequently used. 
\begin{proposition}
Let $P$ be a (non-empty) totally ordered poset. Then 
$\widetilde{H}_n(P)=0$ for all $n\in\Z$. 
\end{proposition}
\begin{proof}
If $P$ is totally ordered, then $\Delta(P)$ is a simplex. Therefore, 
all reduced homology groups vanish. 
\end{proof}

\subsection{Tensor product chain complexes}

\label{sec:tensor}

Next we settle some notations on tensor products of chain complexes 
((\cite[2.7.1]{Wei}). 
Let $(C_{\bullet}^1,\partial_{\bullet}^1), 
(C_{\bullet}^2,\partial_{\bullet}^2), \cdots , 
(C_{\bullet}^m,\partial_{\bullet}^m)$ be chain complexes. 
The chain complex
$\Tot(C_{\bullet}^1 \otimes \cdots \otimes C_{\bullet}^m)_{\bullet}$ 
is defined by 
\begin{equation*}
\Tot(C_{\bullet}^1 \otimes \cdots \otimes C_{\bullet}^m)_n 
:=\bigoplus_{i_1+\cdots +i_m=n}\left( C_{i_{1}}^1 \otimes\cdots\otimes 
C_{i_{m}}^m \right), 
\end{equation*}
and the boundary map is defined by 
$
\sum_{i_1+\cdots+i_m=n}\partial_{i_1, \cdots ,i_{m}}, 
$
where 
\[
\partial_{i_1, \cdots ,i_{m}}: 
C_{i_1}^1\otimes\cdots\otimes C_{i_m}^m\longrightarrow 
\bigoplus_{j_1+\cdots+j_m=n-1}
C_{j_1}^1\otimes\cdots\otimes C_{j_m}^m 
\]
is defined by 
\[
\partial_{i_1,\cdots, i_m}:= 
\bigoplus_{h=1}^m
\left((-1)^{i_1+\cdots+i_{h-1}}\id_{C_{i_1}^1} \otimes \cdots \otimes \partial_{i_{h}}^h\otimes\cdots \otimes \id_{C_{i_m}^m} \right).
\]
The chain complex $\Tot(C_{\bullet}^1 \otimes \cdots \otimes C_{\bullet}^m)_{\bullet}$ is called tensor product chain complex of $(C_{\bullet}^1,\partial_{\bullet}^1), (C_{\bullet}^2,\partial_{\bullet}^2), \cdots ,(C_{\bullet}^m,\partial_{\bullet}^m)$.
The homology $H_n(\Tot(C_{\bullet}^1 \otimes \cdots \otimes C_{\bullet}^m)_{\bullet})$ is denoted by $H_n(C_{\bullet}^1 \otimes \cdots \otimes C_{\bullet}^m)$.
We also have the following associativity. 
\[
\Tot(\Tot(C_\bullet^1\otimes C_\bullet^2 )_\bullet\otimes C_\bullet^3)_\bullet
\simeq
\Tot(C_\bullet^1\otimes \Tot(C_\bullet^2 \otimes C_\bullet^3)_\bullet)_\bullet
\simeq
\Tot(C_\bullet^1\otimes C_\bullet^2 \otimes C_\bullet^3)_\bullet. 
\]

\begin{theorem}
\label{thm:Kunneth formula for complexes}
\textup{(\cite{Wei}, Theorem 3.6.3 (K\"{u}nneth formula for complexes))} 
If $C_{\bullet}^i$ is a chain complex of free $\Z$-modules, then 
\begin{flalign*}
H_n(C_{\bullet}^1 \otimes \cdots \otimes C_{\bullet}^{m+1})
\simeq
\bigoplus_{p+q=n} & H_p( C_{\bullet}^1\otimes\cdots\otimes C_{\bullet}^m )
\otimes
H_q( C_{\bullet}^{m+1}) \nonumber \\ \oplus
&\bigoplus_{p+q=n-1} \Tor_1(H_p(C_{\bullet}^1 \otimes \cdots \otimes C_{\bullet}^{m}),H_q(C_{\bullet}^{m+1})).
\end{flalign*}
\end{theorem}
\begin{proof}
By the associativity,
\begin{equation*}
H_n(C_{\bullet}^1 \otimes \cdots \otimes C_{\bullet}^{m+1})
\simeq
H_n(\Tot(C_{\bullet}^1 \otimes \cdots \otimes C_{\bullet}^m)_{\bullet} \otimes C_{\bullet}^{m+1} ).
\end{equation*}
By K\"{u}nneth formula, 
\begin{flalign*}
H_n(C_{\bullet}^1 \otimes \cdots \otimes C_{\bullet}^{m+1})
\simeq&
H_n(\Tot(C_{\bullet}^1 \otimes \cdots \otimes C_{\bullet}^m)_{\bullet} \otimes C_{\bullet}^{m+1} ) \\ 
\simeq&
\bigoplus_{p+q=n} H_p( C_{\bullet}^1\otimes\cdots\otimes C_{\bullet}^m )
\otimes
H_q( C_{\bullet}^{m+1}) \\ &\oplus
\bigoplus_{p+q=n-1} \Tor_1(H_p(C_{\bullet}^1 \otimes \cdots \otimes C_{\bullet}^{m}),H_q(C_{\bullet}^{m+1})).
\end{flalign*}
\end{proof}

\subsection{The interval decomposition}
\label{sec:ordercpx}

Let $F=(a, b)$ and $|F|=\ell$. 
Suppose $\bm{x}=(x_0, \cdots ,  x_n)\in P_n(X)$ is a geodesically simple 
chain. 
Then $\varphi(\bm{x})=F$ if and only if $x_0=a, x_n=b$ and $|\bm{x}|=\ell$. 
It is also equivalent that 
$x_0=a, x_n=b$ and $x_1, \dots, x_{n-1}\in I_X(a, b)$ form a chain 
$x_1<x_2<\cdots<x_{n-1}$ (of length $n-2$) in the interval poset $I_X(a, b)$. 
By comparing the definitions of boundary maps for magnitude chain complex 
(Definition \ref{def:magnitudechaincpx}) and 
that of order complex, the map 
\[
\langle x_0, x_1, \dots, x_{n-1}, x_n\rangle\longmapsto 
\langle x_1, \dots, x_{n-1}\rangle
\]
gives an isomorphism (up to sign of the boundary operators) 
\begin{equation}
B^{F}_\bullet(X)
\stackrel{\simeq}{\longrightarrow}
C_{\bullet-2}(I_X(a, b)), 
\end{equation}
of chain complexes. In what follows, 
we abbreviate $C_{\bullet}(I_X(a, b))$ as 
$C_{\bullet}(I(a, b))$. 

More generally, we have the following. 
\begin{theorem}
Let $F=(a_0,\cdots,a_m)\in P_{m}(X)$ and $\ell =|F|$. 
Then there is  an isomorphism of chain complexes 
\begin{equation*}
B_{\bullet}^{F}(X)
\simeq
\Tot(C_{\bullet}(I(a_0,a_1))\otimes \cdots \otimes C_{\bullet}(I(a_{m-1},a_m)))_{\bullet-2m}, 
\end{equation*}
as chain complexes (up to sign of boundary operators). 
\end{theorem}

\begin{proof}
The homomorphism $\phi_n$ : $B_n^{F}(X)
\rightarrow
\Tot(C_{\bullet}(I(a_0,a_1))\otimes \cdots \otimes C_{\bullet}(I(a_{m-1},a_m)))_{n-2m}$ is defined as follows (see Figure \ref{fig:decomp}):
\begin{figure}[htbp]
\begin{picture}(350,70)(0,-20)
\thicklines

\put(0,20){$\phi_9($}

\put(30,10){\circle*{4}}
\put(25,0){$a_0$}

\put(30,10){\line(1,3){10}}
{\color{white} 
\put(35,25){\circle*{4}}
}
\put(35,25){\circle{4}}

\put(40,40){\circle*{4}}
\put(35,47){$a_1$}

\put(40,40){\line(2,-3){20}}
{\color{white} 
\multiput(45,32.5)(5,-7.5){3}{\circle*{4}}
}
\multiput(45,32.5)(5,-7.5){3}{\circle{4}}

\put(60,10){\circle*{4}}
\put(55,0){$a_2$}

\put(60,10){\line(0,1){30}}
{\color{white} 
\multiput(60,23)(0,8){2}{\circle*{4}}
}
\multiput(60,23)(0,8){2}{\circle{4}}

\put(60,40){\circle*{4}}
\put(40,40){\circle*{4}}
\put(55,47){$a_3$}

\put(75,20){$)=($}

\put(110,10){\circle*{4}}
\put(105,0){$a_0$}

\put(110,10){\line(1,3){10}}
{\color{white} 
\put(115,25){\circle*{4}}
}
\put(115,25){\circle{4}}

\put(120,40){\circle*{4}}
\put(115,47){$a_1$}

\put(130,20){$)\otimes ($}

\put(160,40){\circle*{4}}
\put(155,47){$a_1$}

\put(160,40){\line(2,-3){20}}
{\color{white} 
\multiput(165,32.5)(5,-7.5){3}{\circle*{4}}
}
\multiput(165,32.5)(5,-7.5){3}{\circle{4}}

\put(180,10){\circle*{4}}
\put(175,0){$a_2$}

\put(190,20){$)\otimes ($}

\put(230,10){\circle*{4}}
\put(225,0){$a_2$}

\put(230,10){\line(0,1){30}}
{\color{white} 
\multiput(230,23)(0,8){2}{\circle*{4}}
}
\multiput(230,23)(0,8){2}{\circle{4}}

\put(230,40){\circle*{4}}
\put(225,47){$a_3$}

\put(250,20){$)$}

\put(85, -20){$\in C_0(I(a_0, a_1))\otimes C_2(I(a_1, a_2))\otimes 
C_1(I(a_2, a_3))$}

\end{picture}
     \caption{Definition of $\phi_n$}\label{fig:decomp}
\end{figure}

\begin{flalign*}
\phi_n&(\langle a_0,x_0^1,\cdots,x_{n_1}^1, a_1,x_0^2,\cdots, x_{n_2}^2, a_2 ,\cdots a_m \rangle) \\
&=\langle x_0^1,\cdots, x_{n_1}^1 \rangle \otimes \langle x_0^2,\cdots,x_{n_2}^2 \rangle \otimes \cdots \otimes \langle x_0^m ,\cdots , x_{n_{m}}^m \rangle.
\end{flalign*}
Then a straightforward computation shows that $(\phi_n)_n$ gives an isomorphism 
of chain complexes. 
\end{proof}

\begin{corollary}
\label{cor:between}
Let $F=(a_0, a_1, \dots, a_m)$. Then 
\begin{equation*}
H_n^{F}(X)\simeq 
H_{n-2m}(\Tot(C_{\bullet}(I(a_0,a_1))\otimes \cdots \otimes 
C_{\bullet}(I(a_{m-1},a_m))).
\end{equation*}
\end{corollary}

\begin{corollary}
\label{cor:vanish}
Let $F=(a_0, a_1, \dots, a_m)$ and $\ell=|F|$. 
Assume that there exists an $i\in\{1, \dots, m\}$ such that 
$I_X(a_{i-1}, a_i)$ is non-empty and totally ordered. 
Then $H_n^{F}(X)=0$ for any $n\in\Z$. 
\end{corollary}
\begin{proof}
Since $C_{\bullet}(I_X(a_{i-1}, a_i))$ is $\Z$-free, acyclic, bounded below 
chain complex, 
by K\"unneth formula, $H_n^{F}(X)=0$ for any $n\in\Z$. 
\end{proof}

\section{Computations of magnitude homology groups}

We exhibit several computations of magnitude homology groups. 

\subsection{Geodetic metric spaces} 
\label{sec:geodetic}
Let $X$ be a geodetic metric space. The magnitude homology of $X$ with grading 
$0<\ell<m_X$ can be precisely described by using the following. 
\begin{definition}
\label{def:thin}
Let $F=(a_0, a_1, \dots, a_m)\in P_m(X)$. $F$ is a \emph{thin} frame 
if $\varphi(F)=F$ and $I_X(a_{i-1}, a_i)=\emptyset$ for any $i=1, \dots, m$. 
\end{definition}
The next result is a generalization of \cite[Theorem 7.25]{Lei-Shu}. 
\begin{theorem}
\label{thm:thinframebase}
Let $X$ be a geodetic space, and $0<\ell<m_X$. Then  
\[
H_n^{\Sigma, \ell}(X)\simeq
\bigoplus_{F\in P_n^\ell(X)\ 
\mbox{\tiny is a thin frame}}\Z\cdot\langle F\rangle
\]
\end{theorem}

\begin{proof}
First apply Theorem \ref{thm:Fdec} (3), and obtain a 
direct sum decomposition into $H_n^{F}(X)$. 
If $F$ is not a thin frame, by Corollary \ref{cor:vanish}, 
$H_n^{F}(X)=0$. Hence we may assume $F$ is a thin frame. 
Let $F=(a_0, \dots, a_m)$. By Corollary \ref{cor:between}, 
\[
H_n^{F}(X)=
H_{n-2m}((C_{\bullet}(I(a_0, a_1))\otimes\cdots\otimes 
C_{\bullet}(I(a_{m-1}, a_m)))_{\bullet}). 
\]
Since $F$ is a thin frame, 
$I_X(a_{i-1}, a_i)=\emptyset$ for all $i=1, \dots, m$. 
Thus we have 
\[
H_k((C_{\bullet}(I(a_0, a_1))\otimes\cdots\otimes 
C_{\bullet}(I(a_{m-1}, a_m)))_{\bullet})=
\left\{
\begin{array}{cc}
\Z, &\mbox{ if }k=-m\\
0, &\mbox{otherwise}. 
\end{array}
\right.
\]
Therefore, $H_n^{F}(X)\simeq \Z$ if and only if 
$n=m$ (and otherwise $H_n^{F}(X)=0$). 
\end{proof}

\begin{corollary}
Let $X\subset\R^N$ be a metric space with Euclidean metric. 
\begin{itemize}
\item[(1)] 
If $X$ is a convex, then 
$H_n^{\Sigma, \ell}(X)=0$ for all $n>0$ and $\ell>0$. 
\item[(2)] 
If $X$ is an open subset of $\R^N$, then 
$H_n^{\Sigma, \ell}(X)=0$ for all $n>0$ and $\ell>0$. 
\end{itemize}
\end{corollary}

\begin{proof}
In both cases, $I_X(a, b)\neq\emptyset$ for any $a\neq b$. Thus 
there do not exist thin frames. (Note that $m_X=\infty$.) 
\end{proof}

\subsection{The diameter of a hole}
\label{sec:diameter}

Now we assume $X$ is geodetic and has no $4$-cuts (i.e. $m_X=\infty$). 

\begin{definition}
Denote by $h_X\in\R_{\geq 0}$ the sup of the distance $d(a, b)$ 
of pairs $(a, b)$ which has not a point between them, namely, 
\[
h_X:=\sup\{d(a, b)\mid a, b\in X, a\neq b, I_X(a, b)=\emptyset\}. 
\]
If there is not a pair $a\neq b$ such that $I_X(a, b)=\emptyset$, 
we set $h_X=0$. 
\end{definition}

\begin{example}
Let $D_r^N\subset\R^N$ be a closed ball of radius $r>0$. 
Let $X=\R^N-\Int(D_r^N)$, where $\Int(D_r^N)$ is the interior of the ball. 
Then $h_X=2r$, which is the diameter of the hole. 
\end{example}

\begin{definition}
Let $k>0$. Define the \emph{vanishing threshold} $\nu_k(X)\geq 0$ 
as follows. 
If $H_k^{\Sigma, \ell}(X)\neq 0$ for some $\ell>0$, then 
\[
\nu_k(X):=
\sup
\{
\ell
\mid
\ell>0, 
H_k^{\Sigma, \ell}(X)\neq 0\}. 
\]
If $H_k^{\Sigma, \ell}(X)= 0$ for any $\ell>0$, $\nu_k(X):=0$. 
\end{definition}

\begin{theorem}
\label{thm:hole}
Assume that $X$ is geodetic and has no $4$-cut. 
Let $k>0$. Then 
\begin{itemize}
\item[(1)] 
$\nu_k(X)=k\cdot h_X$. 
\item[(2)] 
$H_k^{\Sigma, \nu_k(X)}(X)\neq 0$ if and only if 
there exist $a, b\in X$ such that $d(a, b)=h_X$ and $I_X(a, b)=\emptyset$. 
\end{itemize}
\end{theorem}
\begin{proof}
(1) We first prove that $H_k^{\Sigma, \ell}(X)=0$ if $\ell>k\cdot h_X$. 
By Theorem \ref{thm:thinframebase}, 
it is sufficient to show that there does not exist a thin frame of length 
$\ell$. Let $F=(a_0, \cdots, a_k)\in P_k(X)$ with $|F|=\ell$. 
Since 
\[
d(a_0, a_1)+\cdots+d(a_{k-1}, a_k)=\ell>k\cdot h_X, 
\]
there exists $1\leq i\leq k$ such that $d(a_{i-1}, a_i)>h_X$. 
By the definition of $h_X$, $I_X(a_{i-1}, a_i)\neq\emptyset$. 
Hence by Corollary \ref{cor:vanish}, $H_k^{F}(X)=0$ and 
we have $\nu_k(X)\leq k\cdot h_X$. 

Next we show $\nu_k(X)\geq k\cdot h_X$. Let $\ell<k\cdot h_X$. 
Then by definition of $h_X$, there exist $a, b\in X$ such that 
$\frac{\ell}{k}<d(a, b)\leq h_X$ with $I_X(a, b)=\emptyset$. 
Set $\ell'=k\cdot d(a, b)$ and 
\[
F=\left\{
\begin{array}{cl}
(a, b, a, \dots, b)&\mbox{ if $k$ is odd}, \\
(a, b, a, \dots, a)&\mbox{ if $k$ is even}. 
\end{array}
\right. 
\]
Then $F$ is a thin frame of length $\ell'$, and we have 
$H_k^{F}(X)\neq 0$. 
Therefore, $\ell<\ell'\leq\nu_k(X)$, and we have $\nu_k(X)\geq k\cdot h_X$. 

(2) The sufficiency is obvious from the proof of (1). For the rest, 
suppose there do not exist $a, b\in X$ such that $d(a, b)=h_X$ with 
$I(a, b)=\emptyset$. We may assume $k>1$. 
We shall prove that there do not exist a 
thin frame of length $\nu_k(X)$. Let $F=(a_0, \dots, a_k)$ be a frame 
with 
\[
|F|=d(a_0, a_1)+\cdots+d(a_{k-1}, a_k)=\nu_k(X). 
\]
By the assumption there exists an $i$ with 
$d(a_{i-1}, a_i)>\frac{\nu_k(X)}{k}=h_X$ (otherwise, 
$d(a_{i-1}, a_i)<h_X$ for all $i$, and $|F|$ turns out to be 
strictly less than $\nu_k(X)$). Since $I_X(a_{i-1}, a_i)\neq\emptyset$, 
$F$ can not be a thin frame. 
\end{proof}
\begin{remark}
A related result was discussed \cite[Proposition 10]{hep-wil} for 
a metric space defined by a graph. 
\end{remark}

\begin{example}
Let $X=\R^N\smallsetminus\Int(D_r^N)$. Let $\ell>0$. Then 
\[
H_k^{\Sigma, \ell}(X)=
\left\{
\begin{array}{cl}
0&\mbox{ if $\ell>2kr$}\\
\neq 0&\mbox{ if $\ell\leq 2kr$}. 
\end{array}
\right.
\]
Hence $\nu_k(X)=2kr$, and we have 
\[
2r=\frac{\nu_k(X)}{k}=
\sup
\left\{
\frac{\ell}{k}; \ell>0, H_k^{\Sigma, \ell}(X)\neq 0
\right\}. 
\]
\end{example}

\subsection{Embedding of ranked poset homology}
\label{sec:embed}

In \S \ref{sec:ordercpx}, we have seen the relation 
between order complex of posets and magnitude homology. 
In this subsection, we see that order homology of a ranked 
poset can be embedded into the magnitude homology group 
of certain metric spaces. In particular, the homology 
group (with $\Z$ coefficients) of a compact $C^\infty$ manifold 
can be realized as a submodule of a magnitude homology 
of a finite metric space. 

The following is straightforward. 
\begin{lemma}
\label{lem:characterize}
Let $\bm{x}=(x_0, \dots, x_n)$ with $n\geq 2$ 
and $F=(a_0, a_1)$. 
Assume $|\bm{x}|=|F|=\ell$. 
Then the following are equivalent. 
\begin{itemize}
\item[(i)] 
$x_0=a_0$ and $x_n=a_1$. 
\item[(ii)] 
$\bm{x}\in P_n^F(X)$. 
\item[(iii)] 
$\bm{x}\smallsetminus\{x_i\}\in P_{n-1}^F(X)$ 
for any $1\leq i\leq n-1$. 
\item[(iv)] 
$\bm{x}\smallsetminus\{x_i\}\in P_{n-1}^F(X)$ 
for some $1\leq i\leq n-1$. 
\end{itemize}
\end{lemma}

\begin{theorem}
\label{thm:emb}
Let $F=(a_0, a_1)$ with $\ell=|F|>0$. 
The natural map $H_n^F(X)\longrightarrow H_n^{\Sigma, \ell}(X)$ is 
injective. 
\end{theorem}

\begin{proof}
The set $P_n^\ell(X)$ is a disjoint union of 
$P_n^F(X)$ and $P_n^\ell(X)\smallsetminus P_n^F(X)$. 
If $\bm{x}=(x_0, \dots, x_n)\in P_n^\ell(X)\smallsetminus P_n^F(X)$, 
then by Lemma \ref{lem:characterize}, $\bm{x}\smallsetminus\{x_i\}
\notin P_{n-1}^F(X)$ for any $1\leq i\leq n-1$. 
Hence the decomposition induces a direct sum decomposition of 
chain complex 
\[
B_\bullet^\ell(X)=
B_\bullet^F(X)\oplus
\bigoplus_{\bm{x}\in 
P_\bullet^\ell(X)\smallsetminus P_\bullet^F(X)}
\Z\cdot\langle\bm{x}\rangle. 
\]
From this decomposition, it follows the injectivity of 
the embedding. 
\end{proof}

Recall (\cite[\S 3.1]{sta-EC1}) 
that a poset $P$ is said to be a ranked poset if 
every maximal chain of $P$ has the same (finite) length. 
Denote by $\widehat{P}$ the poset obtained from $P$ by 
adjoining a $\widehat{0}$ (the minimum) and 
$\widehat{1}$ (the maximum). 

Let $P$ be a ranked poset. Consider the shortest path 
metric of the Hasse diagram of $\widehat{P}$. In other words, 
for $a, b\in \widehat{P}$, $d(a, b)$ is the minimum $n$ 
such that there exists a sequence $a=x_0, x_1, \dots, x_n=b$ 
with the property that for each $i$ either $x_i$ covers $x_{i-1}$ or 
$x_{i-1}$ covers $x_i$. Then the interval poset between 
$\widehat{0}$ and $\widehat{1}$ is 
\[
I_{\widehat{P}}(\widehat{0}, \widehat{1})\simeq
P
\]
as posets. Theorem \ref{thm:emb} enables us to conclude the following. 
\begin{corollary}
\begin{itemize}
\item[(1)] 
Let $P$ be a ranked poset, $\ell=d(\widehat{0}, \widehat{1})$ 
in $\widehat{P}$. Then there exists an embedding of abelian groups 
$H_{n-2}(\Delta(P))\hookrightarrow H_n^{\Sigma, \ell}(\widehat{P})$ 
for $n\geq 1$. 
\item[(2)] 
Let $M$ be a compact $C^\infty$-manifold. Then there exist 
a finite metric space $X$ and $\ell>0$ such that $H_n^{\Sigma, \ell}(X)$ contains 
a subgroup isomorphic to $H_{n-2}(M, \Z)$ for each $n>0$. 
\item[(3)] 
There exist a graph $X$ and positive integers $n, \ell>0$ such that 
$H_n^{\Sigma, \ell}(X)$ has torsion. 
\end{itemize}
\end{corollary}
\begin{proof}
(1) 
Take the frame $F=(\widehat{0}, \widehat{1})$ in $\widehat{P}$. 
Then apply Theorem \ref{thm:emb}. 

(2) Choose a triangulation of $M$ and denote by $P$ the face 
poset. Since $M$ is a manifold, every simplex is a face of a top 
dimensional simplex. Hence, $P$ is a ranked poset. 
Then the order complex of $P$ is the barycentric subdivision 
of the original triangulation (\cite[\S 10.3.5]{koz}). The assertion 
follows from (1). 

(3) Take $M=\mathbb{RP}^2$. (In this case, $n=3$ and $\ell=4$.) 
\end{proof}

\subsection{Circle with geodesic metric}
\label{sec:circle}

Let $S^1 \subset \R^2$ be the circumference with radius $r$. 
Consider the metric defined by geodesics. For example, the distance 
between antipodal points is $\pi r$. It is easily seen that 
$m_{S^1}=\pi r$. Some of the magnitude homology groups of $S^1$ are 
computed as follows. 

\begin{itemize}
\item[(1)] 
If $\ell<\pi r$ and $n>0$, then $H_n^{\Sigma, \ell}(S^1)=0$. 
\item[(2)] 
If $\ell=\pi r$, then 
\[
H_n^{\Sigma, \ell}(S^1)=
\left\{
\begin{array}{cc}
\Z^{\oplus S^1}, & \mbox{ if }n=2\\
0, & \mbox{ if }n\neq 2, n>0. 
\end{array}
\right.
\]
\item[(3)] 
If $\ell>\pi r$ and $n=1, 2$, then 
\[
H_n^{\Sigma, \ell}(S^1)=0. 
\]
\end{itemize}
(1) follows from Theorem \ref{thm:Fdec} and Corollary \ref{cor:vanish}. 
Let $a, b\in S^1$. 
Denote by $B_n^{\ell, a, b}(S^1)$ the submodule of $B_n^\ell(S^1)$ generated 
by $\bm{x}=(x_0, x_1, \dots, x_n)\in P_n^\ell(S^1)$ with $x_0=a$ and 
$x_n=b$. Then $B_{\bullet}^{\ell, a, b}(S^1)$ is a subcomplex of 
$B_{\bullet}^{\ell}(S^1)$, and the homology group 
$H_n^{\Sigma, \ell}(S^1)$ is a direct sum of 
$\bigoplus_{a, b}H_n(B_{\bullet}^{\ell, a, b}(S^1))$, where $a, b\in S^1$ 
(proved by arguments similar to the proof of Theorem \ref{thm:Fdec},) 
If $\ell=\pi r$ and $a, b$ are not antipodal, 
then the points $x_0, x_1, \dots, x_n$ are contained in the interior 
of a semicircle. 
Then by Corollary \ref{cor:vanish}, homology group 
$H_n(B_{\bullet}^{\ell, a, b}(S^1))$ vanishes. 
When $a$ and $b$ are antipodal (i. e., $d(a, b)=\pi r$), 
$I(a, b)$ is a disjoint union of two totally ordered posets. 
The reduced homology of the order complex of such a poset is 
$\Z$ at degree $0$ and vanishing otherwise. This concludes (2). 
We omit details for (3). 

\begin{remark}
Recently Gomi \cite{gomi} proved that $H_3^{\Sigma, \ell}(S^1)= 0$ for $\ell>0$. 
\end{remark}

\medskip

\noindent
{\bf Acknowledgements.} 
M. Y. was partially supported by JSPS KAKENHI Grant Numbers 
JP15KK0144, JP16K13741. 
The authors thank Ye Liu, Richard Hepworth and Beno\^it Jubin 
for their comments on the draft.

\end{document}